\documentclass[12pt]{amsart}
\usepackage{latexsym,amsmath,amssymb,amsthm}
\usepackage{graphicx}
\usepackage{tikz}
\usetikzlibrary{calc}
\usepackage[enableskew,vcentermath]{youngtab}
\usepackage{hyperref}
\hypersetup{colorlinks=false}
\input epsf
\newtheorem{theorem}{Theorem}[section]
\newtheorem{proposition}[theorem]{Proposition}
\newtheorem{corollary}[theorem]{Corollary}

\newtheorem{definition}[theorem]{Definition}
\newtheorem{example}[theorem]{Example}

\newtheorem{remark}[theorem]{Remark}

\newtheorem{question}[theorem]{Question}

\newcommand{\esd}{{\rm esd}}

\newcommand{\Hur}{\mathbf{Hur}}
\newcommand{\Lace}{\mathbf{Lace}}
\newcommand{\Toep}{\mathbf{Toep}}
\newcommand{\aA}{{\mathcal A}}
\newcommand{\bB}{{\mathcal B}}
\newcommand{\cC}{{\mathcal C}}

\newcommand{\fF}{{\mathcal F}}

\newcommand{\gG}{{\mathcal G}}

\newcommand{\sS}{{\mathcal S}}

\newcommand{\ba}{{\mathbf a}}
\newcommand{\balpha}{\boldsymbol{\alpha}}

\newcommand{\bM}{{\mathbf M}}
\newcommand{\bN}{{\mathbf N}}
\newcommand{\RR}{{\mathbb R}}
\newcommand{\NN}{{\mathbb N}}
\newcommand{\ZZ}{{\mathbb Z}}
\newcommand{\CC}{{\mathbb C}}
\newcommand{\Rr}{\mathsf{R}}
\newcommand{\Ss}{\mathsf{S}}
\renewcommand{\i}{\mathrm{i}}

\begin{document}

\title[Veronese sections and interlacing matrices]
{Veronese sections and interlacing matrices of 
polynomials and formal power series}

\author{Christos~A.~Athanasiadis}
\address{Department of Mathematics\\
National and Kapodistrian University of Athens\\
Panepistimioupolis\\ 15784 Athens, Greece}
\email{caath@math.uoa.gr}

\author{David G. Wagner}
\address{Department of Combinatorics and 
Optimization\\ University of Waterloo\\
Waterloo, Ontario N2L 3G1\\ Canada}
\email{dgwagner@math.uwaterloo.ca}

\date{November 1, 2026}
\thanks{First author supported by the Hellenic 
Foundation for Research and Innovation (H.F.R.I.) 
under the `2nd Call for H.F.R.I. Research Projects 
to support Faculty Members \& Researchers' 
(Project Number: HFRI-FM20-04537).}

\begin{abstract}
The concept of a fully interlacing matrix 
of formal power series with real coefficients is 
introduced. This concept extends and strengthens 
that of an interlacing sequence of real-rooted 
polynomials with nonnegative coefficients, in the 
special case of row and column matrices. The fully 
interlacing property is shown to be preserved under 
matrix products, flips across the reverse diagonal 
and Veronese sections of the power series involved. 
These results and their corollaries generalize, 
unify and simplify several results which have 
previously appeared in the literature. An 
application to the theory of uniform triangulations 
of simplicial complexes is included.
\end{abstract}

\maketitle

\section{Introduction}
\label{sec:intro}

Real-rooted polynomials, especially those with 
nonnegative coefficients, have been studied intensely 
within various mathematical disciplines, including 
algebraic, enumerative and geometric combinatorics
\cite{Bra15, Bre89, Bre94, Sta89}. 
Two important players in this study are the theory of 
interlacing of polynomials \cite[Section~8]{Bra15} 
\cite{Fi06} and the Veronese construction for formal 
power series \cite{BS10, BW09, Jo18, Jo22, Zh20}. 
The present work arose from an effort to better 
understand the connections between these two 
notions.

To explain the content and motivation for this paper,  
we begin with a few definitions. A polynomial $P(x) 
\in \RR[x]$ is said to be \emph{real-rooted} if either it
is the zero polynomial, or every complex root of $P(x)$
is real. Given two polynomials $P(x), Q(x) \in \RR[x]$, 
we say that $P(x)$ \emph{interlaces} $Q(x)$, and write 
$P(x) \prec Q(x)$, if (a) both $P(x)$ and $Q(x)$ are 
real-rooted and have nonnegative coefficients; and (b)
the roots $\{\eta_i\}$ of $P(x)$ interlace (or alternate 
to the left of) the roots $\{\theta_j\}$ of $Q(x)$, in 
the sense that they can be listed as
\[ \cdots \le \theta_3 \le \eta_2 \le \theta_2 \le 
   \eta_1 \le \theta_1 \le 0. \]
By convention, the zero polynomial interlaces and 
is interlaced by every real-rooted polynomial with 
nonnegative coefficients. 

Given a formal power series $A(x) = \sum_{n \ge 0} 
a_n x^n \in \RR[[x]]$ and integers $0 \le k < r$, 
the \emph{$k$th Veronese $r$-section} (or simply 
\emph{$k$th $r$-section}) of $A(x)$ is defined as 
the formal power series
\begin{equation} \label{eq:def-Veronese}
\Ss_k^{(r)}A(x) = \sum_{n \ge 0} a_{k+rn} x^n. 
\end{equation}
We will often suppress the 
superscript $(r)$ from the notation. Every 
$r$-section of a real-rooted polynomial with 
nonnegative coefficients is real-rooted (see, for 
instance, \cite{AESW} \cite[Theorem~7.65]{Fi06} 
\cite[V~171.4]{PS72}). The question whether the 
$r$-section operators preserve interlacing 
arises naturally: Given that $P(x)$ interlaces $Q(x)$, 
does the $k$th $r$-section of $P(x)$ necessarily 
interlace that of $Q(x)$ for all $r, k$? The aim of
this paper is to give an affirmative answer to this 
question (see Corollary~\ref{cor:main}), extend and 
strengthen the corresponding result in several 
directions (see Section~\ref{sec:veronese}) and show 
that a natural context to study such questions is 
provided by the class of AESW series (equivalently, 
that of P\'olya frequency sequences); see 
Section~\ref{sec:AESW}. The latter is characterized 
by the fundamental notion of total positivity.

The main new construction introduced is that of 
an interlacing matrix of a matrix $\aA$ of formal 
power series (see Definition~\ref{def:lace}). In 
the special case that $\aA$ is a column, 
it extends to power series the concept of 
generalized Hurwitz matrix, introduced for 
sequences of polynomials by Goodman and 
Sun~\cite{GS04} (see 
Remark~\ref{rem:genH}). The new concept of 
interlacing, termed as full interlacing, is 
defined by the total positivity of that matrix. 
The special cases of row and column matrices 
strengthen and extend to power series the well 
known concept of an interlacing sequence of 
real-rooted polynomials (both concepts reduce 
to the interlacing relation we have already 
discussed for rows and columns of length two). 
Our main results state that the concept of full 
interlacing behaves well with respect to taking 
submatrices, matrix products, flips across the 
reverse diagonal and Veronese sections of the 
entries of the matrix $\aA$ (see 
Theorems~\ref{thm:flip},~\ref{thm:grid} 
and~\ref{thm:veronese}). They simplify, generalize 
and unify several results which have previously 
appeared in the literature.

Our motivation comes from the study of certain 
operators which are of interest in geometric 
combinatorics~\cite{Ath22} (we describe an 
application in Section~\ref{sec:app} and expect 
that our results will find more applications there), 
as well as from the need to develop more tools for 
proving the interlacing of real-rooted polynomials. 
An extension of the ideas of this paper to total 
positivity up to a certain order or level 
\cite{RSW06} and to tensors of higher rank is 
left to the future.

\section{Three old theorems}
\label{sec:three}

This section reviews three classical theorems on 
stability and total positivity, which will serve 
as the main ingredients for our results, and extends
the notions of interlacing and Hurwitz stability 
from polynomials to formal power series. 

Here, 
$\ba = (a_n)_{n \in \NN} = (a_0, a_1, a_2,\dots)$ 
will be a sequence of real numbers and $A = A(x) = 
\sum_{n \ge 0} a_n x^n$ will be its generating 
function. We set $a_n = 0$ for all negative 
integers $n$. We denote by $\aA^\top$ the 
transpose of a matrix $\aA$. 

\subsection{The Aissen--Edrei--Schoenberg--Whitney 
(AESW) Theorem} 
\label{sec:AESW}

The \emph{Toeplitz matrix} of $A$ is defined as 
$\Toep[A] = (a_{v-u})$; it is indexed by $\ZZ \times 
\ZZ$, with matrix indexing conventions, where $u$ 
indexes rows and $v$ indexes columns. Here is a 
small piece:
\[ \Toep[A] = \begin{pmatrix} 
   \cdot & \cdot & \cdot & \cdot & \cdot & \cdot \\
   \cdot & a_0 & a_1 & a_2 & a_3 & \cdot \\
   \cdot &   0 & a_0 & a_1 & a_2 & \cdot \\
   \cdot &   0 &   0 & a_0 & a_1 & \cdot \\
   \cdot &   0 &   0 &   0 & a_0 & \cdot \\
   \cdot & \cdot & \cdot & \cdot & \cdot & \cdot 
	\end{pmatrix} \]
The sequence $\ba$ is said to be a \emph{P\'olya 
frequency sequence} if $\Toep[A]$ is \emph{totally 
positive} (TP), meaning that every finite square 
submatrix has nonnegative determinant. The 
Aissen--Edrei--Schoenberg--Whitney Theorem 
\cite{AESW,ASW,E1,E2} classifies the power series 
$A(x) \in \RR[[x]]$ for which this happens. They 
are exactly the series of the form \\

\mbox{} \hfill $\displaystyle
A(x) = c x^n \exp(\gamma x) \, 
\frac{\prod_i(1+\alpha_i x)}{\prod_j(1-\beta_j x)}$
\hfill (AESW) \\ 

\noindent
for $n \in \NN$ and nonnegative real numbers $c, 
\alpha_i, \beta_j, \gamma$ such that $\sum_i 
\alpha_i$ and $\sum_j \beta_j$ are finite (we refer
to them as AESW series, in this paper).
We note that a polynomial is AESW if and only if 
it has nonnegative coefficients and only real 
roots and that a series $A(x) \in \RR[[x]]$ with 
constant term $A(0) \ne 0$ is AESW if and only 
if $A(-x)^{-1}$ is AESW.

\subsection{The Hurwitz stability criterion}
The \emph{Hurwitz matrix} of $A$ is also indexed by 
$\ZZ\times\ZZ$ and defined as $\Hur[A] = (a_{2v-u})$.
Here is a small piece:
\[ \Hur[A] = \begin{pmatrix} 
   \cdot & \cdot & \cdot & \cdot & \cdot & \cdot \\
   \cdot & a_0 & a_2 & a_4 & a_6 & \cdot \\
   \cdot & 0   & a_1 & a_3 & a_5 & \cdot \\
   \cdot & 0   & a_0 & a_2 & a_4 & \cdot \\
   \cdot & 0   & 0   & a_1 & a_3 & \cdot \\
   \cdot & 0   & 0   & a_0 & a_2 & \cdot \\
   \cdot & \cdot & \cdot & \cdot & \cdot & \cdot 
    \end{pmatrix} \]

A polynomial $R(x) \in \RR[x]$ is called 
\emph{Hurwitz} (or \emph{Hurwitz stable}) if it has 
nonnegative coefficients and every root of $R(x)$ has 
nonpositive real part. The Hurwitz stability criterion 
\cite{HT12, Ke82} (see also \cite{AGT18, As70, Dy14, 
Ho03}) asserts that this happens if and only if 
$\Hur[R]$ is TP.

More generally, we say that the power series $C(x) \in 
\RR[[x]]$ is \emph{Hurwitz} precisely when $\Hur[C]$ 
is TP. Notice that the Hurwitz matrix $\Hur[C]$ consists 
of two Toeplitz matrices which interleave each other,
namely $\Toep[A]$ and $\Toep[B]$, where $C(x) = 
B(x^2) + xA(x^2)$ (thus, $A(x)$ and $B(x)$ are the odd 
and even parts of $C(x)$, respectively). As a 
consequence, if a power series $C(x) = B(x^2) + x 
A(x^2) \in \RR[[x]]$ is Hurwitz, then $A(x)$ and $B(x)$ 
are both AESW. In particular, if a polynomial $R(x) = 
Q(x^2) + x P(x^2) \in \RR[x]$ is Hurwitz, then both 
$P(x)$ and $Q(x)$ have nonnegative coefficients and 
only real roots. 

\subsection{The Hermite--Biehler Theorem}
\label{sec:HBT}

Recall from the introduction that a polynomial $P(x) 
\in \RR[x]$ is said to \emph{interlace} a polynomial 
$Q(x) \in \RR[x]$, written $P(x) \prec Q(x)$, if both 
$P(x)$ and $Q(x)$ are AESW and the roots of $P(x)$ 
interlace those of $Q(x)$. The Hermite--Biehler 
Theorem \cite{Bi79, He56} (see also \cite{Ga60}) 
implies that a polynomial $R(x) = Q(x^2) + x P(x^2) 
\in \RR[x]$ is Hurwitz if and only if $P(x) \prec 
Q(x)$. The same theorem relates the interlacing 
between $P(x)$ and $Q(x)$ to the notion of 
(univariate) \emph{stability} of the complex 
polynomial $Q(z) + \i P(z)$, meaning that all roots 
have nonpositive imaginary part; see 
\cite[Theorem~6.3.4]{RS02} \cite[p.~57]{Wa11} for 
a precise formulation.

We may again generalize the concept of interlacing
(cf.~\cite{Dy14}) from polynomials to series as 
follows. Given power series $A(x), B(x) \in 
\RR[[x]]$ with $A(x) = 
\sum_{n \ge 0} a_n x^n$ and $B(x) = \sum_{n \ge 0} 
b_n x^n$, we introduce the \emph{interlacing matrix}
\[ \Lace \begin{pmatrix} A(x) \\ B(x) \end{pmatrix} 
   := \Hur[B(x^2)+xA(x^2)] = \begin{pmatrix}
\cdot & \cdot & \cdot & \cdot & \cdot & \cdot \\
\cdot & b_0 & b_1 & b_2 & b_3 & \cdot \\
\cdot & 0   & a_0 & a_1 & a_2 & \cdot \\
\cdot & 0   & b_0 & b_1 & b_2 & \cdot \\
\cdot & 0   & 0   & a_0 & a_1 & \cdot \\
\cdot & 0   & 0   & b_0 & b_1 & \cdot \\
\cdot & \cdot & \cdot & \cdot & \cdot & \cdot 
   \end{pmatrix} \]
and say that $A(x)$ \emph{interlaces} $B(x)$, 
denoted $A(x) \prec B(x)$, if 
$\Lace(A(x) \, B(x))^\top$ is TP. This implies 
that both $A(x)$ and $B(x)$ are AESW. For 
polynomials, this agrees with the notion of 
interlacing we have already discussed. 

\section{Fully interlacing matrices}
\label{sec:main}

This section introduces the concept of a fully 
interlacing matrix of formal power series. We 
begin with the important special case of 
interlacing sequences. It will be convenient 
for us to think of sequences of formal power 
series as column vectors. 

\subsection{Fully interlacing sequences}
\label{sec:int-seq}

Let $\aA = (A_0(x) \, A_1(x) \, \cdots \, 
A_{p-1}(x))^\top$ be a sequence of formal power 
series in $\RR[[x]]$ (this indexing convention 
will be convenient for us later). We call $\aA$ 
\emph{pairwise interlacing} if $A_i(x) \prec A_j
(x)$ for all $0 \le i < j \le p-1$.

Extending the definition of the interlacing 
matrix of two power series from 
Section~\ref{sec:HBT}, we set
\[ \Lace \begin{pmatrix} A(x) \\ B(x) \\ C(x) 
   \end{pmatrix} =
   \begin{pmatrix}
\cdot & \cdot & \cdot & \cdot & \cdot & \cdot & 
        \cdot \\
\cdot & c_0 & c_1 & c_2 & c_3 & c_4 & \cdot \\
\cdot & 0   & a_0 & a_1 & a_2 & a_3 & \cdot \\
\cdot & 0   & b_0 & b_1 & b_2 & b_3 & \cdot \\
\cdot & 0   & c_0 & c_1 & c_2 & c_3 & \cdot \\
\cdot & 0   & 0   & a_0 & a_1 & a_2 & \cdot \\
\cdot & \cdot & \cdot & \cdot & \cdot & \cdot & 
        \cdot \end{pmatrix} \]
for $A(x) = \sum_{n \ge 0} a_n x^n$, $B(x) = 
\sum_{n \ge 0} b_n x^n$, $C(x) = \sum_{n \ge 0} 
c_n x^n \in \RR[[x]]$ and similarly for 
$\Lace(\aA)$ for every sequence $\aA = (A_0(x) 
\, A_1(x) \, \cdots \, A_{p-1}(x))^\top$ of power 
series in $\RR[[x]]$. We call $\aA$ \emph{fully 
interlacing} if the matrix $\Lace(\aA)$ is TP. 
\begin{remark} \label{rem:genH} \rm
When the $A_i(x)$ are polynomials with real 
coefficients, $\Lace(\aA)$ is called a 
\emph{generalized Hurwitz matrix}. These 
matrices were studied by Goodman and Sun~\cite{GS04}; 
see also \cite[Section~4.8]{P10} and~\cite{HKK16}. 
Holtz, Khrushchev and Kushel~\cite{HKK16} showed 
via a generalized Euclidean algorithm and continued 
fraction analysis that if $P(x)$ is a polynomial 
with real coefficients and 
\[ \Lace(\Ss_{r-1}^{(r)} P(x) \ \cdots \ 
   \Ss_1^{(r)} P(x) \ \Ss_0^{(r)} P(x))^\top \] 
is TP (see also Section~\ref{sec:veronese}), 
then $P(x)$ has no root in the sector 
$\{ z \in \CC: |\arg(z)| < \pi/r \}$. As in 
the introduction, $\Ss_k^{(r)}P(x)$ here stands 
for the $k$th Veronese $r$-section of $P(x)$.
\qed
\end{remark}

Our next two statements follow directly from 
the definition.

\begin{proposition}[Heredity]
Every subsequence of a fully interlacing sequence 
of formal power series in $\RR[[x]]$ is fully 
interlacing. In particular, every fully 
interlacing sequence is pairwise interlacing.
\end{proposition}

The next statement is the primary means of producing
AESW series, and especially real-rooted polynomials, 
in many applications. The proof is postponed until 
Example~\ref{ex:grid} in the sequel, since it can be 
deduced from more general results.
\begin{proposition}[Convexity] \label{lma-AESWcone}
If $\aA = (A_0(x) \, A_1(x) \, \cdots \, 
A_{p-1}(x))^\top$ is a fully interlacing sequence 
of formal power series in $\RR[[x]]$, then 
\[ A_0(x) \prec \sum_{i=0}^{p-1} \lambda_i A_i(x) 
          \prec A_{p-1}(x) \]
for all nonnegative real numbers $\lambda_0, 
\lambda_1,\dots,\lambda_{p-1}$. In particular, 
every series in the cone generated by $A_0(x), 
A_1(x),\dots,A_{p-1}(x)$ is AESW.
\end{proposition}

\begin{example} \rm
Pairwise interlacing sequences which are not fully 
interlacing can be constructed as follows. Let $a, 
b, c, d, t$ be positive real numbers such that $a 
\le c$ and $b \le d$ and consider the polynomials 
$P(x) = t+x$, $Q(x) = (b+x)(d+x)$ and $R(x) = 
(a+x)(c+x)$. Then, the sequence $(P(x) 
\ Q(x) \ R(x))^\top$ is pairwise interlacing if 
and only if $a \le b \le t \le c \le d$. The 
interlacing matrix 
$\Lace (P(x) \ Q(x) \ R(x))^\top$ has
\[ \det \begin{pmatrix} 
   t & 1 & 0 \\ bd & b+d & 1 \\ ac & a+c & 1 
	 \end{pmatrix} = t(b+d-a-c) \, - \, (bd-ac) \]
as a minor, hence $(P(x) \ Q(x) \ R(x))^\top$ is 
not fully interlacing for $t < (bd-ac)/(b+d-a-c)$. 
When $(a, b, c, d) = (1, 2, 3, 4)$, any value 
$2 \le t < 5/2$ provides such an example.
\end{example}

\subsection{Fully interlacing grids.}

Let $\aA$ be a $p \times q$ matrix of power series.
Then, $\Lace[\aA]$ is obtained by interleaving the 
Toeplitz matrices of the series in $\aA$ in a 
natural way, for instance:
\begin{equation} \label{eq:2x2} 
   \Lace \begin{pmatrix}
   A(x) & C(x) \\ B(x) & D(x) \end{pmatrix} = 
	 \begin{pmatrix}
   \cdot & \cdot & \cdot & \cdot & \cdot & \cdot & 
	 \cdot \\
\cdot & a_0 & c_0 & a_1 & c_1 & a_2 & \cdot \\
\cdot & b_0 & d_0 & b_1 & d_1 & b_2 & \cdot \\
\cdot &  0  &  0  & a_0 & c_0 & a_1 & \cdot \\
\cdot &  0  &  0  & b_0 & d_0 & b_1 & \cdot \\
\cdot &  0  &  0  &  0  &  0  & a_0 & \cdot \\
\cdot & \cdot & \cdot & \cdot & \cdot & \cdot & 
\cdot \end{pmatrix} \end{equation}

Formally, we have the following definition. 

\begin{definition} \label{def:lace} \rm
Let $\aA = (A_{ij}(x))$ be a $p \times q$ matrix of 
power series 
\begin{equation} \label{eq:Aij}
A_{ij}(x) = \sum_{n \ge 0} a_{ij}(n) x^n \in 
            \RR[[x]], 
\end{equation}
with entries indexed by $i \in \{0, 1,\dots,p-1\}$
and $j \in \{0, 1,\dots,q-1\}$. Let $\Lace(\aA) = 
(M_{uv})$ be the $\ZZ \times \ZZ$ matrix defined as 
follows. For $(u,v) \in \ZZ \times \ZZ$, 
let $(u',i)$ and $(v',j)$ be the unique pairs of 
integers such that $u = pu'+i$ and $v = qv'+j$ 
with $0 \le i < p$ and $0 \le j < q$. Then,
\[ M_{uv} = a_{ij}(v'-u') = [x^{v'-u'}]A_{ij}(x) 
   \]
is the coefficient of $x^{v'-u'}$ in $A_{ij}(x)$. 
We call $\aA$ \emph{fully interlacing} if 
$\Lace(\aA)$ is TP.
\end{definition}

\begin{remark} \rm
Let $\balpha_n$ be the $p \times q$ matrix
$(a_{ij}(n))$ for every $n \in \NN$, so that $\aA = 
\sum_{n \ge 0} \balpha_n x^n$. Then, $\Lace(\aA)$ has
the block decomposition
\[ \Lace(\aA) = \begin{pmatrix} 
   \cdot & \cdot & \cdot & \cdot & \cdot & \cdot \\
   \cdot & \balpha_0 & \balpha_1 & \balpha_2 & 
	         \balpha_3 & \cdot \\
   \cdot &   O & \balpha_0 & \balpha_1 & \balpha_2 
	       & \cdot \\
   \cdot &   O &   O & \balpha_0 & \balpha_1 & \cdot \\
   \cdot &   O &   O &   O & \balpha_0 & \cdot \\
   \cdot & \cdot & \cdot & \cdot & \cdot & \cdot 
	\end{pmatrix} \]
in which every block is a $p \times q$ matrix. 
Clearly, this expression specializes to the 
definition of the Toeplitz matrix for $p=q=1$. 
\qed
\end{remark}

\begin{example} \label{ex:def-lace} \rm
Let $A(x), B(x) \in \RR[[x]]$.

(a) A $p \times q$ matrix of constant power series 
is fully interlacing if and only if it is TP.

(b) By the AESW Theorem, the $1 \times 1$ matrix 
$(A(x))$ is fully interlacing if and only if $A(x)$ 
is AESW. 

Similarly, by definition, the $2 \times 1$ matrix 
$(A(x) \ B(x))^\top$ is fully interlacing if and only 
if $A(x) \prec B(x)$. By the Hermite--Biehler Theorem,
this statement holds for polynomials in $A(x), B(x) 
\in \RR[x]$ and the classical notion of interlacing.
\qed
\end{example}

Definition~\ref{def:lace} agrees with our earlier 
notion of fully interlacing sequence when $\aA$ is 
a column matrix and directly implies the following 
statement.
\begin{proposition}[Heredity] 
\label{lem:mat-heredity}
Every submatrix of a fully interlacing matrix of 
formal power series in $\RR[[x]]$ is fully 
interlacing.
\end{proposition}

For fixed $p$ and $q$, the map which assigns 
$\Lace(\aA)$ to $\aA$ is an injection from the 
set of $p \times q$ matrices with entries in
$\RR[[x]]$ to the set of $\ZZ \times \ZZ$ 
matrices with entries in $\RR$. However, for 
different such $p \times q$ matrices $\aA$ and
$\aA^\prime$ it is possible that $\Lace
(\aA^\prime) = (M_{uv}^\prime)$ is a \emph{shift} 
of $\Lace(\aA) = (M_{uv})$, meaning there 
exists $(a, b) \in \ZZ \times \ZZ$ such that 
$M_{uv}^\prime = M_{u+a,v+b}$ for every $(u, v) 
\in \ZZ \times \ZZ$. For instance,
\[ \Lace \begin{pmatrix}
   B(x) & D(x) \\ xA(x) & xC(x) \end{pmatrix} = 
	 \begin{pmatrix}
   \cdot & \cdot & \cdot & \cdot & \cdot & \cdot & 
	 \cdot \\
   \cdot & b_0 & d_0 & b_1 & d_1 & b_2 & \cdot \\ 
   \cdot &   0 &   0 & a_0 & c_0 & a_1 & \cdot \\
   \cdot &  0  &  0  & b_0 & d_0 & b_1 & \cdot \\
   \cdot &  0  &  0  &   0 &   0 & a_0 & \cdot \\
   \cdot &  0  &  0  &  0  &  0  & b_0 & \cdot \\ 
   \cdot & \cdot & \cdot & \cdot & \cdot & \cdot & 
   \cdot \end{pmatrix} \]
is a shift of (\ref{eq:2x2}). Then, $\Lace(\aA)$ 
is TP if and only if $\Lace(\aA^\prime)$ is TP. 
The following statement generalizes the previous 
example. The special case of column matrices 
generalizes the fact that $(P_0(x) \, P_1(x) \, 
\cdots \, P_{n-1}(x))^\top$ is a pairwise 
interlacing sequence of polynomials if and only 
if $(P_1(x) \, \cdots \, P_{n-1}(x) \ 
xP_0(x))^\top$ has the same property.
\begin{proposition} \label{prop:shift}
Let $\aA = (A_{ij}(x))$ be as in 
Definition~\ref{def:lace}. 
\begin{itemize}
\item[(a)]
The matrix $\aA$ is fully interlacing if and only 
if 
\[ \begin{pmatrix} 
   A_{10}(x) & A_{11}(x) & \cdots & A_{1,q-1}(x) 
	 \\ \vdots & \vdots & \vdots & \vdots \\ 
	A_{p-1,0}(x) & A_{p-1,1}(x) & \cdots & A_{p-1,q-1}
	(x) \\
  xA_{00}(x) & xA_{01}(x) & \cdots & xA_{0,q-1}(x)
	\end{pmatrix} \]
is fully interlacing.

\item[(b)]
Similarly for 
\[ \begin{pmatrix} 
   xA_{0,q-1}(x) & A_{00}(x) & \cdots & A_{0,q-2}(x) 
	 \\
   xA_{1,q-1}(x) & A_{10}(x) & \cdots & A_{1,q-2}(x)
	 \\
   \vdots & \vdots & \vdots & \vdots \\ 
	 xA_{p-1,q-1}(x) & A_{p-1,0}(x) & \cdots & 
	 A_{p-1,q-2}(x) 
	\end{pmatrix}. \]
\end{itemize}
\end{proposition}

\begin{proof}
Let $\aA^\prime$ be the given matrix. Then,
$\Lace(\aA^\prime)$ is the shift of $\Lace(\aA)$ 
by the vector $(1,0)$ and $(0,-1)$ in the cases
of parts (a) and (b), respectively, and the proof 
follows.
\end{proof}

The proof shows that, in fact, $\Lace(\aA^\prime)$ 
is a shift of $\Lace(\aA)$ if and only if 
$\aA^\prime$ can be obtained from $\aA$ by 
applying successively the operations described 
in Proposition~\ref{prop:shift} and their 
inverses.

\begin{remark} \rm
The $2 \times 2$ case of 
Definition~\ref{def:lace} corresponds to the 
``interpolating squares'' of \cite{Wa00}.
Although it is not clear how these concepts are 
related, the corresponding notion of 
``interpolating hypercube'' suggests there may
be a concept of fully interlacing tensors of 
higher rank.
\end{remark}

\section{Flips and products}
\label{sec:flips}

This section and the following one include the main 
results of this paper. 

Given a $p \times q$ matrix $\aA = (A_{ij}(x))$, 
we denote by $\aA^\bot$ the $q \times p$ matrix 
obtained by reflecting $\aA$ across its reverse 
diagonal. With indices $0 \le i < p$ and $0 \le j 
< q$, this is the matrix $\aA^\bot = (A^\bot_{ij}
(x))$ defined by setting $A^\bot_{ji}(x) = 
(A_{p-1-i,q-1-j}(x))$; we call it the \emph{flip}
of $\aA$. The flip of a $\ZZ \times \ZZ$ matrix 
$\bM = (M_{uv})$ is defined as the 
$\ZZ \times \ZZ$ matrix $\bM^\bot = (M^\bot_{uv})$ 
such that $M^\bot_{vu} = M_{-1-u,-1-v}$ for every 
$(u,v) \in \ZZ \times \ZZ$.

As already discussed, columns of fully interlacing 
matrices are fully interlacing sequences, in the 
sense of Section~\ref{sec:int-seq}, when read from 
top to bottom. The following result implies the
same for rows, when read from \emph{right to left}. 

\begin{theorem} \label{thm:flip} 
Let $\aA$ be a $p \times q$ matrix of formal power 
series.
\begin{itemize}
\item[(a)]
$\Lace(\aA^\bot) = \Lace(\aA)^\bot$.
\item[(b)]
$\aA^\bot$ is fully interlacing if and only if 
$\aA$ is fully interlacing.
\end{itemize}
\end{theorem}

\begin{proof}
We convert between indices with the following 
notation: $u, v \in \ZZ$ and
\begin{align*}
u=pu'+i \ \ & \text{with} \ \ 0 \le i<p \ 
              \text{and} \ u'\in\ZZ, \\
v=qv'+j \ \ & \text{with} \ \ 0 \le j<q \ 
              \text{and} \ v'\in\ZZ.
\end{align*}
Note that
\begin{align*}
-1-u &= p(-1-u')+(p-1-i)\ \ \text{and} \ \\
-1-v &= q(-1-v') +(q-1-j),
\end{align*}
so that $(-1-u)'=-1-u'$ and $(-1-v)'=-1-v'$ (modulo 
$p$ and $q$, respectively).

Let $\aA = (A_{ij}(x))$, so that $\aA^\bot = 
(A^\bot_{ji}(x)) = (A_{p-1-i,q-1-j}(x))$. For $u, v 
\in \ZZ$ the $(v, u)$-entry of $\Lace(\aA)^\bot$ 
equals
\begin{align*}
(\Lace(\aA)^\bot)_{vu} &=
\Lace(\aA)_{-1-u,-1-v} = [x^{(-1-v)'-(-1-u)'}]
A_{p-1-i,q-1-j}(x) \\ &=
[x^{u'-v'}] \aA^\bot_{ji}(x) = (\Lace(\aA^\bot))_{vu}.
\end{align*}
This proves part (a). Part (b) follows from part (a) 
and the fact that $\det(L^\bot) = \det(L)$ for every 
finite square matrix $L$.
\end{proof}

The following example shows that the notion of fully
interlacing matrix does not behave as well with 
respect to transposing the matrix. 
\begin{example} \rm 
For positive real numbers $a, b, c, d$ consider 
the interlacing matrix of linear forms
\[ \Lace(\aA) = \Lace \begin{pmatrix} 
   a+x & c+x \\ b+x & d+x \end{pmatrix} 
   = \begin{pmatrix} 
\cdot & \cdot & \cdot & \cdot & \cdot & \cdot & 
        \cdot \\
\cdot &  a  &  c  &  1  &  1  &  0  & \cdot \\
\cdot &  b  &  d  &  1  &  1  &  0  & \cdot \\
\cdot &  0  &  0  &  a  &  c  &  1  & \cdot \\
\cdot &  0  &  0  &  b  &  d  &  	1  & \cdot \\
\cdot &  0  &  0  &  0  &  0  &  a  & \cdot \\
\cdot & \cdot & \cdot & \cdot & \cdot & \cdot &  
        \cdot \end{pmatrix} \]
For this to be TP, its $2 \times 2$ minors must be 
nonnegative, so $b \le a \le c$, $b \le d \le c$ 
and $bc \le ad$. The only way for both $\aA$ and 
$\aA^\top$ to be interlacing is that $a=b=c=d$.
\end{example}

\begin{theorem} \label{thm:grid}
Let $\aA$ and $\bB$ be $p \times t$ and $t \times q$ 
matrices of formal power series, respectively.
\begin{itemize}
\item[(a)]
$\Lace(\aA\bB) = \Lace(\aA) \Lace(\bB)$.

\item[(b)]
If $\aA$ and $\bB$ are fully interlacing, then 
$\aA\bB$ is fully interlacing.
\end{itemize}
\end{theorem}

\begin{proof}
We let $\Lace(\aA) = \bM = (M_{uw})$, $\Lace(\bB)
= \bN = (N_{wv})$, $\aA = (A_{ik}(x))$, $\bB = 
(B_{kj}(x))$ and $\cC = \aA \bB = (C_{ij}(x))$. We 
convert between indices with the following notation:
$u, v, w \in \ZZ$ and
\begin{align*}
u=pu'+i \ \ & \text{with} \ \ 0 \le i < p \ 
              \text{and} \ u' \in \ZZ, \\
v=qv'+j \ \ & \text{with} \ \ 0 \le j < q \ 
              \text{and} \ v' \in \ZZ, \\
w=tw'+k \ \ & \text{with} \ \ 0 \le k < t \ 
              \text{and} \ w' \in \ZZ.
\end{align*}
Then, for $u, v \in \ZZ$, the $(u, v)$-entry of 
$\Lace(\aA\bB)$ equals
\begin{align*}
\Lace(\aA\bB)_{uv}
& = [x^{v'-u'}]C_{ij}(x) = 
    [x^{v'-u'}] \sum_{0 \le k<t} A_{ik}(x) B_{kj}(x) 
\\ &= \sum_{0 \le k<t} \sum_{w' \in \ZZ} 
      \left([x^{w'-u'}] A_{ik}(x) \right) 
			\left([x^{v'-w'}] B_{kj}(x) \right) \\
   &= \sum_{w \in \ZZ} M_{uw} N_{wv} = (\bM\bN)_{uv}.
\end{align*}
This proves part (a). Part (b) follows from part (a) 
and the fact that products of TP matrices are TP.
\end{proof}

\begin{example} \label{ex:grid} \rm
(a) The special case $t = q = 1$ of 
Theorem~\ref{thm:grid} asserts that if $A(x) \in 
\RR[[x]]$ is AESW and $\aA$ is a fully interlacing 
sequence of formal power series in $\RR[[x]]$, then 
so is the sequence obtained from $\aA$ by 
multiplying every entry with $A(x)$.

(b) We recall (see the comment before 
Theorem~\ref{thm:flip}) that the full interlacing 
of a sequence $(A_0(x) \, A_1(x) \, \cdots \, 
A_{t-1}(x))^\top$ of formal power series in 
$\RR[[x]]$ is equivalent to the full interlacing 
of $(A_{t-1}(x) \, \cdots \, A_1(x) \, 
A_0(x))$. Thus, the special case $p = q = 1$ of 
Theorem~\ref{thm:grid} asserts that if 
\begin{align*}
\aA &= (A_0(x) \, A_1(x) \, \cdots \, 
        A_{t-1}(x))^\top 
\\ \bB &= (B_0(x) \, B_1(x) \, \cdots \, 
           B_{t-1}(x))^\top 
\end{align*}
are fully interlacing sequences of power series in 
$\RR[[x]]$, then 
\[ \aA^\perp \bB = \sum_{i=0}^{t-1} A_{t-1-i}(x) B_i(x) 
   = A_0(x) B_{t-1}(x) + A_1(x) B_{t-2}(x) + \cdots + 
	 A_{t-1}(x) B_0(x) \] 
is AESW. This should be compared with 
\cite[Lemma~7.8.3]{Bra15}, which states that in the 
case that the $A_i(x)$ and $B_i(x)$ are polynomials
with nonnegative coefficients,
it suffices to assume that the two sequences are 
pairwise interlacing. The special case in which the 
$B_i(x)$ are constants is equivalent to the last 
statement in Proposition~\ref{lma-AESWcone}.

(c) Suppose that the $2 \times q$ matrix 
\[ \begin{pmatrix}
   A_{q-1}(x) & \cdots & A_1(x) & A_0(x)  \\  
   B_{q-1}(x) & \cdots & B_1(x) & B_0(x)
\end{pmatrix} \]
of formal power series in $\RR[[x]]$ is fully 
interlacing. According to Theorem~\ref{thm:grid}, 
multiplying on the right by a column vector 
$(\lambda_{q-1} \, \cdots \, \lambda_1 \, 
\lambda_0)^\top$ of 
nonnegative real numbers yields a $2 \times 1$ 
fully interlacing matrix. This means that 
\[ \sum_{i=0}^{q-1} \lambda_i A_i(x) \prec 
	 \sum_{i=0}^{q-1} \lambda_i B_i(x) \]
for all nonnegative real numbers $\lambda_0, 
\lambda_1,\dots,\lambda_{q-1}$. More generally, 
by the same argument,
\[ \sum_{i=0}^{q-1} C_i(x) A_{q-1-i}(x) \prec 
	 \sum_{i=0}^{q-1} C_i(x) B_{q-1-i}(x) \]
for every fully interlacing sequence $(C_0(x) 
\, C_1(x) \, \cdots \, C_{q-1}(x))^\top$ of 
formal power series in $\RR[[x]]$.

(d) By Example~\ref{ex:def-lace} and 
Theorem~\ref{thm:grid}, the full interlacing 
property of a matrix $\aA$ of power series is 
preserved under multiplication with totally 
positive matrices (for instance, with Toeplitz 
matrices of AESW power series) of appropriate 
size. In the special case that $\aA$ is a 
column vector with polynomial entries, 
\cite[Theorem~3.7]{Fi06} states that full 
interlacing can be replaced by pairwise 
interlacing.

For example, multiplying a fully interlacing 
column matrix $(A_0(x) \, A_1(x) \, \cdots \, 
A_{p-1}(x))^\top$ on the left with the TP matrix 
\[ \begin{pmatrix} 
   1 & 0 & \cdots & 0 \\ 
	 \lambda_0 & \lambda_1 & \cdots & \lambda_{p-1} 
	 \\ 0 & 0 & \cdots & 1 \end{pmatrix} \]
shows that the triple $(A_0(x), \sum_{i=0}^{p-1} 
\lambda_i A_i(x), A_{p-1}(x))$ is also fully 
interlacing and proves 
Proposition~\ref{lma-AESWcone}. Multiplying a 
fully interlacing matrix $\aA$ of power series 
on the left with a row vector (respectively, on 
the right by a column vector) with all entries 
equal to one shows that the column sums of $\aA$ 
read from right to left (respectively, the row 
sums of $\aA$ read from top to bottom) are fully 
interlacing sequences. Similarly, multiplying 
$\aA$ on the left (respectively, on the right) 
with an upper-triangular square matrix of 
appropriate size with all entries on or above 
the main diagonal equal to one yields a fully 
interlacing matrix whose rows (respectively, 
columns) are the partial sums of the rows 
(respectively, columns) of $\aA$ read from 
bottom to top (respectively, from left to 
right). 
\qed
\end{example}

\begin{remark} \rm
The diagonals and reverse diagonals of fully 
interlacing matrices need not be fully interlacing.
Indeed, by Theorem~\ref{thm:grid}, the matrices 
\[ \begin{pmatrix} 1 \\ x+1 \end{pmatrix} 
   \begin{pmatrix} x & x+1 \end{pmatrix} = 
	 \begin{pmatrix} x & x+1 \\ x^2+x & (x+1)^2 
	 \end{pmatrix} \]
and 
\[ \begin{pmatrix} 1 \\ x \end{pmatrix} 
   \begin{pmatrix} x & 1 \end{pmatrix} = 
	 \begin{pmatrix} x & 1 \\ x^2 & x 
	 \end{pmatrix} \]
are fully interlacing but their diagonal and 
reverse diagonal elements, respectively, do not
interlace each other.
\end{remark}

\begin{remark} \rm
According to Theorem~\ref{thm:grid}, 
multiplication on the left by a $p \times q$ 
fully interlacing matrix of formal power series 
in $\RR[[x]]$ preserves the full interlacing 
property of $q \times 1$ column-matrices (sequences 
of length $q$) of formal power series. We do not 
know if the converse holds, namely whether every 
$p \times q$ matrix of formal power series which 
preserves the full interlacing property of 
sequences of length $q$ must be fully 
interlacing. 
\end{remark}

\section{Veronese sections}
\label{sec:veronese}

We recall that $\Ss_k = \Ss_k^{(r)}$ stands for 
the $k$th Veronese $r$-section operator on 
formal power series. 
Given a $p \times q$ matrix $\aA = (A_{ij}(x))$ 
of formal power series and a positive integer $r$, 
we denote by $\Ss_k^{(r)} \aA$ the $p \times q$ 
matrix $( \Ss_k^{(r)} A_{ij}(x))$ which is 
obtained by applying $\Ss_k^{(r)}$ to every 
entry of $\aA$. Clearly, $(\Ss_k^{(r)} \aA)^\top 
= \Ss_k^{(r)} \aA^\top$ and $(\Ss_k^{(r)} 
\aA)^\perp = \Ss_k^{(r)} \aA^\perp$. 

We set
\[ \sS^{(r)} \aA = (\Ss_0^{(r)} \aA \ \, 
   \Ss_1^{(r)} \aA \ \cdots \ 
	 \Ss_{r-1}^{(r)} \aA) \ \ \ \ \text{and} 
	 \ \ \ \  
   \sS^{(r) \bot} \aA = \begin{pmatrix} 
   \Ss_{r-1}^{(r)} \aA \\ \vdots \\ 
	 \Ss_1^{(r)} \aA \\ \Ss_0^{(r)} \aA 
	 \end{pmatrix}; \]
these are $p \times rq$ and $rp \times q$ matrices,
respectively. We note that 
\[ (\sS^{(r)} \aA)^\perp = \begin{pmatrix} 
   \Ss_{r-1}^{(r)} \aA^\perp \\ \vdots \\ 
	 \Ss_1^{(r)} \aA^\perp \\ \Ss_0^{(r)} \aA^\perp 
	 \end{pmatrix} = \sS^{(r) \bot} \aA^\perp, \]
which is an $rq \times p$ matrix, and that 
$(\sS^{(r)} \aA)^\perp \ne \sS^{(r) \bot} \aA$ 
in general (even when $p=q$).

\begin{example} \rm
We have
\begin{align*} 
\sS^{(r)} A(x) &= (\Ss_0 A(x) \ \Ss_1 A(x) \ 
                   \cdots \ \Ss_{r-1} A(x)), \\ 
\sS^{(r) \bot} A(x) &= (\Ss_0 A(x) \ \Ss_1 A(x) 
                   \ \cdots \ \Ss_{r-1} A(x))^\bot 
\end{align*}
for every $A(x) \in \RR[[x]]$ and
\begin{align*}
\sS^{(2)} (A(x) \ B(x)) &= (\Ss_0 A(x) \ \Ss_0 
            B(x) \ \Ss_1 A(x) \ \Ss_1 B(x)) \\
\sS^{(2) \bot} (A(x) \ B(x)) & = \begin{pmatrix} 
   \Ss_1 A(x) & \Ss_1 B(x) \\ \Ss_0 A(x) & \Ss_0 
	 B(x) \end{pmatrix} = \begin{pmatrix} 
   \Ss_0 B(x) & \Ss_1 B(x) \\ \Ss_0 A(x) & \Ss_1 
	 A(x) \end{pmatrix}^\perp \\ & = \left(
	\sS^{(2)} \begin{pmatrix} B(x) \\ A(x) 
	\end{pmatrix} \right)^\perp
	\end{align*}
for all $A(x), B(x) \in \RR[[x]]$. 
\end{example}

\begin{remark} \label{rem:veronese} \rm
Given positive integers $r$ and $s$, let us 
use the shorthand $\Rr_i = \Ss_i^{(r)}$ and 
$\Ss_j = \Ss_j^{(s)}$.

(a) We have $\Rr_i \Ss_j A(x) = \Ss_{j+si}^{(rs)} 
A(x)$ for every $A(x) \in \RR[[x]]$ and all 
$0 \le i < r$, $0 \le j < s$. Indeed, if $A(x) =
\sum_{n \ge 0} a_n x^n$, then 
\begin{align*}
\Rr_i \Ss_j A(x) &= 
\Rr_i \Ss_j \sum_{n \ge 0} a_n x^n = 
\Rr_i \sum_{n \ge 0} a_{j + sn} x^n = 
\sum_{n \ge 0} a_{j + s(i + rn)} x^n \\ &= 
\sum_{n \ge 0} a_{j + si + rsn} x^n = 
\Ss_{j+si}^{(rs)} \sum_{n \ge 0} a_n x^n =
\Ss_{j+si}^{(rs)} A(x).
\end{align*}

(b) We have 
\begin{align*} 
   \sS^{(r) \bot} \sS^{(s)} \aA &= \begin{pmatrix}
   \Rr_{r-1} \Ss_0 \aA & \Rr_{r-1} \Ss_1 \aA & 
	 \cdots & \Rr_{r-1} \Ss_{s-1} \aA \\ 
   \vdots & \vdots & \cdots & \vdots \\
		\Rr_1 \Ss_0 \aA & \Rr_1 \Ss_1 \aA & \cdots &
		\Rr_1 \Ss_{s-1} \aA \\
		\Rr_0 \Ss_0 \aA & \Rr_0 \Ss_1 \aA & \cdots &
		\Rr_0 \Ss_{s-1} \aA \end{pmatrix}, \\ & \\
		\sS^{(s)} \sS^{(r) \bot} \aA &= \begin{pmatrix}
   \Ss_0 \Rr_{r-1} \aA & \Ss_1 \Rr_{r-1} \aA & 
	 \cdots & \Ss_{s-1} \Rr_{r-1} \aA \\ 
   \vdots & \vdots & \cdots & \vdots \\
		\Ss_0 \Rr_1 \aA & \Ss_1 \Rr_1 \aA & \cdots &
		\Ss_{s-1} \Rr_1 \aA \\
		\Ss_0 \Rr_0 \aA & \Ss_1 \Rr_0 \aA & \cdots &
		\Ss_{s-1} \Rr_0 \aA \end{pmatrix},
\end{align*}
where $\Rr_i \Ss_j = \Ss_{j+si}^{(rs)}$ and $\Ss_j 
\Rr_i = \Ss_{i+rj}^{(rs)}$. Both these matrices are 
$rp \times sq$, where $\aA$ is $p \times q$. 
Moreover, 
\[ \sS^{(r)} \sS^{(s)} \aA = \sS^{(rs)} \aA = 
   (\Ss_0^{(rs)} \aA \ \, \Ss_1^{(rs)} \aA \ \cdots 
	  \ \Ss_{rs-1}^{(rs)} \aA) \]
is the $p \times rsq$ matrix obtained from the first 
matrix by reading each row from left to right, 
starting from the bottom row and continuing to 
the top, or from the second matrix by reading each
column from bottom to top, staring from the leftmost
column and continuing to the right. In particular, 
the operators $\sS^{(r)}$ and $\sS^{(s)}$ commute.
\end{remark}

\begin{theorem} \label{thm:veronese}
Let $\aA = (A_{ij}(x))$ be a $p \times q$ matrix 
of formal power series. If $\aA$ is fully 
interlacing, then so are the matrices 
$\sS^{(r)} \aA$ and $\sS^{(r) \bot} \aA$ for every
positive integer $r$. 
\end{theorem}

\begin{proof}
Let us adopt the notation of Equation~(\ref{eq:Aij})
for the $A_{ij}(x)$. We first consider $\sS^{(r)} 
\aA = (\Ss_0^{(r)} \aA \ \, \Ss_1^{(r)} \aA \ \cdots 
\ \Ss_{r-1}^{(r)} \aA)$, which is a $p \times rq$ 
matrix with entries 
\[ \Ss_k^{(r)} A_{ij}(x) = \sum_{n \ge 0} 
               a_{ij}(k+rn) x^n \]
for $0 \le k < r$ and $0 \le i < p$, $0 \le j < q$.
Given $(u, v) \in \ZZ \times \ZZ$, we set 
\begin{align*}
u=pu'+i \ \ & \text{with} \ \ 0 \le i < p \ 
              \text{and} \ u' \in \ZZ, \\
v=rqv'+j' \ \ & \text{with} \ \ 0 \le j' < rq \ 
              \text{and} \ v' \in \ZZ, \\
j'=kq+j \ \ & \text{with} \ \ 0 \le k < r \ 
              \text{and} \ 0 \le j < q.
\end{align*}
Then, the $(u,v)$-entry of $\Lace(\sS^{(r)} \aA)$
is equal to
\begin{align*}
\left( \Lace(\sS^{(r)} \aA) \right)_{uv} &= 
       [x^{v' - u'}] \, (\sS^{(r)} \aA)_{ij'}(x) = 
			 [x^{v' - u'}] \, \Ss_k^{(r)} A_{ij} (x) \\ 
			 &= a_{ij} (k + r(v'-u')) = 
			 [x^{k+rv'-ru'}] \, A_{ij}(x). 
\end{align*}
Since $v=(k+rv')q+j$, we have shown that the 
the $(u,v)$-entry of $\Lace(\sS^{(r)} \aA)$
is equal to the $(ru',v)$-entry of $\Lace(\aA)$
for every $(u, v) \in \ZZ \times \ZZ$. This 
means that $\Lace(\sS^{(r)} \aA)$ is the 
submatrix of $\Lace(\aA)$ consisting of all 
rows with indices congruent to $0, 1,\dots,p-1$ 
modulo $rp$. A similar argument shows that 
$\sS^{(r) \bot} \aA$ is the submatrix of
$\Lace(\aA)$ consisting of all columns with 
indices congruent to $0, 1,\dots,q-1$ modulo 
$rq$. In particular, $\Lace(\sS^{(r)} \aA)$ 
and $\Lace(\sS^{(r) \bot} \aA)$ are TP as 
submatrices of a TP matrix.  
\end{proof}

\begin{remark} \rm
All $r$-sections $\Ss_k^{(r)} A(x)$ of an AESW 
power series $A(x) \in \RR[[x]]$ are AESW as 
well; see \cite[Theorem~7]{AESW}. 
Theorem~\ref{thm:veronese} implies the stronger 
result that $\sS^{(r)} A(x)$ is a fully 
interlacing sequence. The weaker 
result, stating that $\sS^{(r)} A(x)$ is pairwise 
interlacing for every AESW polynomial $A(x) \in
\RR[x]$, appears as \cite[Theorem~7.65]{Fi06}.
\end{remark}

\begin{example} \rm
Let 
\begin{align*}
Q_n(x) &= \Ss^{(2)}_0 (1+x)^n = \sum_{k \ge 0} 
          {n \choose 2k} x^k \\
P_n(x) &= \Ss^{(2)}_1 (1+x)^n = \sum_{k \ge 0} 
          {n \choose 2k+1} x^k
\end{align*}
for $n \in \NN$. Since $(1+x)^n \prec (1+x)^{n+1}$, 
Theorem~\ref{thm:veronese} implies that the 
matrices
\[ \begin{pmatrix} 
   Q_n(x) & P_n(x) \\ Q_{n+1}(x) & P_{n+1}(x) 
	 \end{pmatrix} \ \ \ \text{and} \ \ \ 
	 \begin{pmatrix} 
   P_n(x) \\ P_{n+1}(x) \\ Q_n(x) \\ Q_{n+1}(x) 
	 \end{pmatrix} \]
are fully interlacing for every $n \in \NN$.
\qed
\end{example}

As mentioned in the introduction, the question 
whether part (b) of the following statement is 
valid motivated our investigations.
\begin{corollary} \label{cor:main}
\begin{itemize}
\item[(a)]
If $A(x), B(x) \in \RR[[x]]$ are formal power series 
and $A(x) \prec B(x)$, then the sequence
\[ (\Ss_0^{(r)} B(x) \ \Ss_0^{(r)} A(x) \ \Ss_1^{(r)} 
   B(x) \ \Ss_1^{(r)} A(x) \ \cdots \ \Ss_{r-1}^{(r)} 
	 B(x) \ \Ss_{r-1}^{(r)} A(x))^\bot \]
is fully interlacing.

\item[(b)]
If $P(x), Q(x) \in \RR[x]$ are polynomials and $P(x) 
\prec Q(x)$, then the sequence
\[ (\Ss_0^{(r)} Q(x) \ \Ss_0^{(r)} P(x) \ \Ss_1^{(r)} 
   Q(x) \ \Ss_1^{(r)} P(x) \ \cdots \ \Ss_{r-1}^{(r)} 
	 Q(x) \ \Ss_{r-1}^{(r)} P(x))^\bot \]
is pairwise interlacing. In particular, $\Ss_k^{(r)} 
P(x) \prec \Ss_k^{(r)} Q(x)$ for $0 \le k < r$. 
\end{itemize}
\end{corollary}

\begin{proof}
This follows by applying Theorem~\ref{thm:veronese}
to $\sS^{(r) \bot} \aA$ for $\aA = (A(x) \ B(x))^\top$
and $\aA = (P(x) \ Q(x))^\top$, respectively.
\end{proof}

\begin{example} \rm
Let 
\begin{align*}
P(x) &= (x+1)(x+3)(x+5)(x+7)(x+9)(x+11)(x+13) \\
Q(x) &= x(x+2)(x+4)(x+6)(x+8)(x+10)(x+12)
\end{align*}
and $r=3$. Then, $P(x) \prec Q(x)$ and 
\begin{align*}
\Ss_2 P(x) &= 973x + 177331 \\
\Ss_2 Q(x) &= 700x + 56448 \\
\Ss_1 P(x) &= x^2 + 10045x + 264207 \\
\Ss_1 Q(x) &= x^2 + 5880x + 46080 \\
\Ss_0 P(x) &= 49x^2 + 57379x + 135135 \\
\Ss_0 Q(x) &= 42x^2 + 25984x.
\end{align*}
Rounding to two decimal places, the roots of these 
polynomials are $-180.25$, $-80.64$, $-10018.62 < 
-26.37$, $-5872.15 < -7.84$, $-1168.64 < -2.35$ and
$-618.66 < 0$, respectively. Hence, the sequence
\[ (\Ss_2 P(x) \ \Ss_2 Q(x) \ \Ss_1 P(x) \ \Ss_1 
    Q(x) \ \Ss_0 P(x) \ \Ss_0 Q(x))^\top \]
is pairwise interlacing, in agreement with 
Corollary~\ref{cor:main} (b).
\qed
\end{example}

Various results can be deduced by combining 
Theorem~\ref{thm:veronese} and the results of 
Section~\ref{sec:flips}; here is an example.
\begin{corollary} 
If $A(x), B(x) \in \RR[[x]]$ are formal power series 
and $A(x) \prec B(x)$, then 
\[ \sum_{k=0}^{r-1} \lambda_k \Ss_k^{(r)} A(x) \prec 
	 \sum_{k=0}^{r-1} \lambda_k \Ss_k^{(r)} B(x) \]
for all nonnegative real numbers $\lambda_0, 
\lambda_1,\dots,\lambda_{r-1}$.
\end{corollary}

\begin{proof}
This follows by applying Example~\ref{ex:grid} (c)
to the matrix $\sS^{(r)} \aA$ for $\aA = (A(x) \ 
B(x))^\top$, which is fully interlacing by 
Theorem~\ref{thm:veronese}.
\end{proof}

\section{An application}
\label{sec:app}

This section describes a concrete application to
the theory of uniform triangulations of simplicial
complexes~\cite{Ath22, AT21}. We assume familiarity
with basic notions about simplicial complexes,
their triangulations and their face enumeration;  
such background can be found 
in~\cite{Bj95, StaCCA}.

A triangulation $\Delta'$ of an $(n-1)$-dimensional 
simplicial complex $\Delta$ is said to be 
\emph{uniform} \cite{Ath22} if for all $0 \le i \le 
j \le n$ and for every $(j-1)$-dimensional face 
$F \in \Delta$, the number of $(i-1)$-dimensional 
faces of the restriction of $\Delta'$ to $F$ depends 
only on $i$ and $j$. We denote this number by 
$f_{ij}$ and say that the triangular array $\fF = 
(f_{ij})_{0 \le i \le j \le n}$ is the 
\emph{$f$-triangle} (of size $n$) associated to 
$\Delta'$ and that $\Delta'$ is an 
\emph{$\fF$-uniform} triangulation of $\Delta$. 
The $h$-polynomial of a simplicial complex 
$\Delta$, denoted $h(\Delta, x)$, provides a 
convenient way to encode its face numbers 
\cite[Section~II.2]{StaCCA}. One of the main 
results of~\cite{Ath22} (see 
\cite[Theorem~4.1]{Ath22}) states that the
$h$-polynomial of an $\fF$-uniform triangulation 
$\Delta'$ of a simplicial complex $\Delta$ depends 
only on $h(\Delta,x)$ and $\fF$ and describes this 
dependence explicitly. We may thus denote 
$h(\Delta',x)$ by $h_\fF(\Delta,x)$. The following 
two families of polynomials play an important role
in the theory:
\begin{align*}
h_\fF(\sigma_m,x), & \ \ \text{for} \ m \in 
                    \{0, 1,\dots,n\}, \\
\theta_\fF(\sigma_m,x) &= h_\fF(\sigma_m, x) - 
                     h_\fF(\partial \sigma_m, x),
			 \ \ \text{for} \ m \in \{0, 1,\dots,n\},
\end{align*}
where $\sigma_m$ stands for the $(m-1)$-dimensional
simplex and $\partial \sigma_m$ is its boundary 
complex. The polynomials $h_\fF(\sigma_m,x)$ have 
nonnegative coefficients and, under some mild 
assumptions, so do the $\theta_\fF(\sigma_m,x)$.
Following \cite[Section~3]{AT21}, we say that $\fF$ 
has the \emph{strong interlacing property} if 

\begin{itemize}
\itemsep=0pt
\item[{\rm (i)}]
$h_\fF(\sigma_m, x)$ is real-rooted for all $2 \le 
m < n$, 

\item[{\rm (ii)}]
$\theta_\fF(\sigma_m, x)$ is either identically
zero, or a real-rooted polynomial of degree $m-1$ 
with nonnegative coefficients which is interlaced 
by $h_\fF(\sigma_{m-1}, x)$, for all 
$2 \le m \le n$.
\end{itemize}
These conditions imply strong real-rootedness 
properties for the $h$-polynomials of $\fF$-uniform 
triangulations of simplicial complexes and their 
symmetric decompositions \cite[Section~6]{Ath22} 
\cite[Section~4]{AT21}, such as the 
real-rootedness of $h_\fF(\Delta,x)$ for every 
$(n-1)$-dimensional Cohen--Macaulay simplicial 
complex $\Delta$ \cite[Theorem~6.1]{Ath22}. It is
thus natural to ask \cite[Question~7.1]{AT21} which
$f$-triangles of uniform triangulations have the 
strong interlacing property. Such $f$-triangles 
include those of barycentric subdivisions, $r$-fold
edgewise subdivisions (for $r \ge n$) and 
$r$-colored barycentric subdivisions 
\cite[Section~7]{Ath22} \cite[Section~5]{AT21}; 
see \cite[Section~5]{Ath22} \cite[Section~3]{MW17} 
and the references given there for the relevant 
definitions and background. 

Given an $f$-triangle $\fF$ of size $n$, we denote 
by $\esd_r(\fF)$ the $f$-triangle of size $n$ 
associated to the $r$-fold edgewise subdivision of 
an $\fF$-uniform triangulation of $\sigma_n$. The 
main result of this section is as follows.
\begin{proposition} \label{prop:app}
Let $\fF$ be the $f$-triangle associated to a 
uniform triangulation of the simplex $\sigma_n$. If
$\fF$ has the strong interlacing property, then so 
does $\esd_2(\fF)$.
\end{proposition}

\begin{proof}
We assume that $\fF$ has the strong interlacing 
property and write $\gG := \esd_2(\fF)$. We then 
have (see, for instance, 
\cite[Equation~(9)]{Ath22})
\[ h_\gG (\Delta, x) = \Ss^{(2)}_0 
   \left( (1+x)^n h_\fF(\Delta, x) \right) \]
for every $(n-1)$-dimensional simplicial complex 
$\Delta$.

Since $\Ss^{(r)}_k$ preserves real-rootedness, 
$h_\gG (\sigma_m, x) = \Ss^{(2)}_0 ( (1+x)^m 
h_\fF (\sigma_m, x))$ is real-rooted for all 
$2 \le m < n$. This verifies condition (i) for 
$\gG$. To verify condition (ii), for $2 \le m 
\le n$ we compute that

\begin{align}
\theta_\gG(\sigma_m, x) & = 
h_\gG(\sigma_m, x) - h_\gG(\partial \sigma_m, 
x) \nonumber \\ & = 
\Ss^{(2)}_0 \left( (1+x)^m h_\fF (\sigma_m, x) 
\right) - \Ss^{(2)}_0 \left( (1+x)^{m-1} 
h_\fF(\partial \sigma_m, x) \right) \nonumber 
\\ & = \Ss^{(2)}_0 \left( (1+x)^{m-1} (h_\fF 
(\sigma_m, x) - h_\fF(\partial \sigma_m, x)) 
\right) + \\ & \ \ \ \ \ \, \Ss^{(2)}_0 \left( 
x (1 + x)^{m-1} h_\fF(\sigma_m, x) \right) 
\nonumber \\ & = \Ss^{(2)}_0 \left( (1+x)^{m-1} 
\theta_\fF(\sigma_m, x) \right) + x \Ss^{(2)}_1 
\left( (1 + x)^{m-1} h_\fF(\sigma_m, x) \right).
\label{eq:sum}
\end{align}
By assumption, $\theta_\fF(\sigma_m, x)$ either
is identically zero, or has nonnegative coefficients 
and degree $m-1$. In the former case, $h_\fF(\sigma_m, 
x) = h_\fF(\partial\sigma_m, x)$ and hence $h_\fF
(\sigma_m, x)$ also has nonnegative coefficients 
and degree $m-1$. Thus, the formula we have reached 
shows that $\theta_\gG(\sigma_m, x)$ has nonnegative 
coefficients and degree $m-1$ as well.

Finally, $h_\fF(\sigma_{m-1}, x)$ interlaces 
$\theta_\fF(\sigma_m, x)$ and $h_\fF(\sigma_m, x)$ 
by assumption and \cite[Theorem~6.1]{Ath22}. As a 
result, $(1+x)^{m-1} h_\fF(\sigma_{m-1}, x)$ 
interlaces both $(1+x)^{m-1} \theta_\fF(\sigma_m, 
x)$ and $(1+x)^{m-1} h_\fF(\sigma_m, x)$ and 
Corollary~\ref{cor:main} implies that 

\begin{align*}
\Ss^{(2)}_1 \left( (1 + x)^{m-1} h_\fF(\sigma_m, x) 
\right) & \prec \Ss^{(2)}_0 \left( (1+x)^{m-1} 
h_\fF(\sigma_{m-1}, x) \right) \\ & = 
h_\gG(\sigma_{m-1}, x) \\ & \prec \Ss^{(2)}_0 \left( 
(1+x)^{m-1} \theta_\fF(\sigma_m, x) \right). 
\end{align*}
Equivalently, $h_\gG(\sigma_{m-1}, x)$ interlaces 
both summands of~(\ref{eq:sum}). Therefore, $h_\gG
(\sigma_{m-1}, x)$ interlaces their sum 
$\theta_\gG(\sigma_m, x)$ and the proof follows.
\end{proof}

\begin{question} 
Does Proposition~\ref{prop:app} continue to hold 
if $\esd_2(\fF)$ is replaced by $\esd_r(\fF)$ for
any $r \ge 2$?
\end{question} 

Let us conclude with a possible connection of the 
concept of full interlacing to the theory of 
uniform triangulations. Given 
an $f$-triangle $\fF$ of size $n$, there exist 
polynomials $p_{\fF,n,k}(x)$ with nonnegative 
integer coefficients, for $k \in \{0, 1,\dots,n\}$,
such that 
\[ h_\fF(\Delta,x) = \sum_{k=0}^n h_k(\Delta) 
                     p_{\fF,n,k}(x) \]
for every $(n-1)$-dimensional simplicial complex
$\Delta$ (\cite[Theorem~4.1]{Ath22}). Here 
$(h_k(\Delta))_{0 \le k \le n}$ is the $h$-vector 
of $\Delta$, so that $h(\Delta, x) = \sum_{k=0}^n
h_k(\Delta)x^k$. Hence, the transformation of the 
$h$-vector of $\Delta$ under
$\fF$-uniform triangulation is linear and the 
associated matrix, say $H_\fF$, in the standard 
basis has the vectors of coefficients of the 
$p_{\fF,n,k}(x)$ as its columns. The matrix 
$H_\fF$ has been shown to be TP for the 
$f$-triangles of barycentric subdivision
\cite{LM26+, MWZ26} and $r$-fold edgewise 
subdivision \cite{MW22}. Since $H_\fF^\top$ is a 
submatrix of $\Lace(\aA_\fF)$, where 
\[ \aA_\fF = (p_{\fF,n,0}(x), 
              p_{\fF,n,1}(x),\dots,p_{\fF,n,n}
							(x)), \]
the total positivity of $H_\fF$ is implied by 
the full interlacing of $\aA_\fF$. Mao and Wang
\cite{MW22} in fact prove that $\aA_\fF$ is fully
interlacing in the case of $r$-fold edgewise 
subdivision. Thus, it seems reasonable to ask 
the following questions.

\begin{question} \label{que:AF1}
Is $\aA_\fF$ fully interlacing when $\fF$ is the 
$f$-triangle of $r$-colored barycentric subdivison? 
\end{question} 

\begin{question} \label{que:AF2}
Under what conditions on $\fF$ is $\aA_\fF$ fully 
interlacing? Does this happen whenever $\fF$ has 
the strong interlacing property?
\end{question} 

Computer data supports an affirmative answer to 
Question~\ref{que:AF1} for $r=1$. A combinatorial 
interpretation of $p_{\fF,n,k}(x)$ when $\fF$ is 
the $f$-triangle of $r$-colored barycentric 
subdivison appears in \cite[Proposition~4.7]{Ath22}. 

\section*{Acknowledgements}
First author supported by the Hellenic 
Foundation for Research and Innovation (H.F.R.I.) 
under the `2nd Call for H.F.R.I. Research Projects 
to support Faculty Members \& Researchers' 
(Project Number: HFRI-FM20-04537). The authors 
thank Peter Br\"and\'en for useful comments and 
references and the anonymous referee for useful
suggestions.

\end{document}